\documentclass{commat}

\usepackage{amsmath,amsthm,amscd,euscript,longtable}
\usepackage{graphicx}
\usepackage{color}
\usepackage{subcaption}
\usepackage[normalem]{ulem}

\def \r{\mathbb R}

\def \z{\mathbb Z}

 \DeclareMathOperator{\iv}{lV}
\DeclareMathOperator{\isin}{lsin}
\DeclareMathOperator{\icos}{lcos}
\DeclareMathOperator{\itan}{ltan}
\DeclareMathOperator{\iarctan}{larctan}

\DeclareMathOperator{\IARCTAN}{IARCTAN}
\DeclareMathOperator{\is}{lS}
\DeclareMathOperator{\il}{l\ell}

\title{
    Multidimensional integer trigonometry
    }

\author{
    John Blackman, Oleg Karpenkov and James Dolan
    }

\authorinfo[
    J. Blackman]{The University of Liverpool, Mathematical Sciences Building, Liverpool, L69 7ZL, United Kingdom}{
    john.blackman@liverpool.ac.uk
    }

\authorinfo[O. Karpenkov]{The University of Liverpool, Mathematical Sciences Building, Liverpool, L69 7ZL, United Kingdom}{
    karpenk@liverpool.ac.uk
    }

\authorinfo[J. Dolan]{The University of Liverpool, Mathematical Sciences Building, Liverpool, L69 7ZL, United Kingdom}{
    james.dolan@liverpool.ac.uk
    }

\abstract{This paper is dedicated to providing an introduction into multidimensional integer trigonometry.
We start with an exposition of integer trigonometry in two dimensions, which was introduced in 2008,
and use this to generalise these integer trigonometric functions to arbitrary dimension.
We then move on to study the basic properties of integer trigonometric functions.
We find integer trigonometric relations
for transpose and adjacent simplicial cones, and for the cones which generate the same simplices.
Additionally, we discuss the relationship between integer trigonometry,
the Euclidean algorithm, and continued fractions.
Finally, we use adjacent and transpose cones to introduce a notion of  best approximations of simplicial cones. In two dimensions, this notion of best approximation coincides with the classical notion of the best approximations of real numbers.
    }



\VOLUME{31}
\NUMBER{2}
\YEAR{2023}
\firstpage{1}
\DOI{ https://doi.org/10.46298/cm.10919}

\begin{document}

\section*{Introduction}

Consider an integer lattice $\z^n\subset \r^n$.
We say that two subsets of $\r^n$ are {\it integer congruent}
if there exists affine integer lattice preserving transformations of $\r^n$ sending one subset to another.
In this paper, we are interested in the invariants of the integer congruence classes of simplicial cones.
These invariants take the form of \textit{integer trigonometric functions}, so called due to the similarities between them and the Euclidean trigonometric functions.

Traditionally, integer geometry has studied questions regarding the interplay of integer lattices and Euclidean geometry
(for example, classical sphere-packing problems~\cite{CS-1999} and view-obstruction problems,
e.g., the lonely runner problem introduced by J.M.~Wills~\cite{Wills1967} in 1967). In this setting, integer lattice invariants have limited usage due to restrictions imposed by Euclidean geometry.
Once we relax these restrictions and consider integer geometry on its own, the lattice invariants take a central role and reveal a rich combinatorial structure.
The integer trigonometric functions that we study in this paper are invariants of integer affine transformations and are closely related to subtractive algorithms (which act as generalisations of the Euclidean algorithm) and multidimensional continued fractions (for instance, the ones introduced by F.~Klein~\cite{Klein1895,Klein1896} in 1895 and
further developed by V.~Arnold~\cite{Arnold2002}).
Integer geometry has applications to a variety of different topics, notably,
the theory of algebraic irrationalities~\cite{Karpenkov2022}, the
theory of generalized Markov numbers~\cite{KarpenkovVanSon}, and Gauss reduction theory~\cite{Manin2002}.

Trigonometric functions are fundamental objects in all aspects of Euclidean geometry, partly due to their relations to dot products, cross products and Pl\"ucker coordinates.
It is natural to expect that trigonometric functions are
natural phenomena in any geometry, in particular, integer geometry.
In 2008, an integer lattice analogue of trigonometric functions was developed for simplicial cones in $\r^2$, see~\cite{itrig1,itrig2}.
Integer trigonometry and the associated continued fractions
have proved to have wider applications to toric
geometry~\cite{Hirzebruch1953,Jung1908} and cuspidal singularities~\cite{Tsuchihashi1983}.
As shown in~\cite{itrig1}, the two-dimensional
integer trigonometric functions impose conditions on toric singularities of a fixed Euler characteristic.

This paper is dedicated to providing a generalisation of integer trigonometric functions in arbitrary dimension. After a short survey of results regarding the two-dimensional case, we introduce the definition of the integer sine of a simplicial cone and note the relationship to the integer volume. Likewise, we introduce the integer arctangent of a simplicial cone and note the relation to Hermite's normal form. Integer cosines and tangents are then derived using the entries of Hermite's normal form.
This leads to a discussion of the basic properties of these multidimensional integer trigonometric functions, such as the relationship between the values of adjacent cones, transpose cones and cones which generate the same simplex.
We develop the connection between the strong best approximations of rational numbers and the transpose and adjacent angles of the corresponding rational angle. This provides the foundation to generalise the notion of strong best approximations to rational cones in any dimension.
The majority of the statements are formulated for the most elementary case, the case of \textit{simple rational cones} (see Section~\ref{Sec3}).
We briefly discuss general cones at the very end of the paper.

{
\noindent
{\bf This paper is organised as follows.}
In Section~\ref{CIT}, we introduce the basic notions of integer trigonometry in two-dimensions and list some known results.
In Section~\ref{Sec2}, we extend these notions to higher dimensions.
Section~\ref{Sec3} provides a discussion of how transposing integer cones and taking adjacent cones affects the integer trigonometric functions.
We also investigate the relations on the integer trigonometric functions of cones which generate the same simplex.
We conclude Section~\ref{Sec3} by introducing a notion of strong best approximations of cones.
In Section~\ref{Sec4}, we prove the main statements of Section~\ref{Sec3}.
We conclude this paper by giving a short discussion of non-simple cones in Section~\ref{A few words on general lattices}.
}

\section{Classical Integer Trigonometry}\label{CIT}

\subsection{Preliminary definitions}

Our main objects of study for this paper are \textit{ordered simplicial cones} in $\r^n$. Before we do this, let us first introduce the notion of a \textit{simplical cone}.

\begin{definition}
Consider a point $x$ and a set of linearly independent vectors $V$ in $\r^n$.
A \textit{$($simplicial$)$ cone} $C(x;V)$ over $V$  with vertex $x$ is the set
$$
\Bigg\{
x+\sum\limits_{v\in V}\lambda_vv
\Big|
\lambda_v\in\r_{\geq0} \hbox{\,\, for $v\in V$}
\Bigg\}.
$$
Any cone generated by a subset of $V$ centered at $x$ is said to be a {\it face} of $\alpha$.
The one-dimensional faces of the cone are referred to as {\it edges}.
\end{definition}

\begin{definition}
Let $A_0,\ldots, A_k$ be a $(k+1)$-tuple of points in $\r^n$ ($k\le n$) where the vectors $v_i=A_0A_i$
are all linearly independent.
The ordered simplicial cone $\angle(A_0;A_1\ldots A_k)$
is the simplicial cone $C(A_0;\{v_1,\ldots, v_k\})$ with a natural ordering of edges
induced by the $(k+1)$-tuple.
\end{definition}

The most interesting cases of $k$-cones is when they are contained in $\r^k$, i.e., when the dimensions are equal.

\begin{remark}
In the two dimensional case we still prefer to use the standard planar notation, i.e., $\angle A_1A_0A_2=\angle (A_0;A_1A_2)$.
We call such cones {\it angles}.
\end{remark}

\begin{remark} For the remainder of the paper, all cones will be considered to be ordered.
\end{remark}

A cone is {\it integer} if its vertex is integer and an integer cone is \textit{rational} if all
its edges contain an integer point distinct from the vertex.
A polyhedron is \textit{integer}, if all of its vertices are integer.

Two points/segments/polyhedra/cones $A$, $B$ in $\r^n$ are \textit{integer congruent} if there exists a matrix $M\in \text{Aff}(n,\z)$ such that $A=M\cdot{B}$. In this case, we write $A\sim{B}$.

We conclude this subsection with a general geometric problem on integer cones (in the spirit of the IKEA problem, see Problem~\ref{IKEA-problem-2d}).

\begin{problem}
Given three rational angles in $\r^3$, is there a $3$-cone whose 2-dimensional faces are given by these three angles? Given $k$ rational angles in $\r^n$, is there a $k$-cone whose 2-dimensional faces are given by these $k$ angles?
\end{problem}

For the remainder of Section~\ref{CIT}, we restrict ourselves to objects in $\r^2$.

\subsection{Invariants of integer geometry}

\subsubsection{Integer length and integer area}
The invariants of integer geometry are typically inherited from the invariants of the corresponding integer sublattices. For instance,
this is the case for integer lengths and integer areas.

\begin{definition}[Integer length]
Let $AB$ be an integer segment and let $L$ be a line through $A$ and $B$.
Denote the lattice of all integer points contained on this line as $\Gamma$, and the sublattice of $\Gamma$ generated by integer multiples of $AB$ as $\Gamma_1$.
The {\it integer length} $\il(AB)$ is the index of $\Gamma_1$ in $\Gamma$, i.e., $\il(AB)=[\Gamma:\Gamma_1]$.
\end{definition}

Combinatorially, the integer length of the segment $AB$ is the number of integer points that lie on the segment $AB$ minus one.
Similarly, we can define an invariant notion of the integer area of a triangle, which leads us to the definition of the integer area of a polygon.

\begin{definition}[Integer area]
Let $\triangle ABC$ be an integer triangle and let $\Gamma_1$ be the lattice generated by integer multiples of $AB$ and $BC$.
The {\it integer area} $\is(\triangle ABC)$ of $\triangle ABC$ is the index of $\Gamma_1$ in $\z^2$, i.e., $\is(\triangle ABC)=[\z^2:\Gamma_1]$.

Let $P$ be an integer polygon and let $T$ be a decomposition of $P$ into integer triangles. Then, the {\it integer area} $\is(P)$ of $P$ is the sum of the integer areas of the triangles in $T$.
\end{definition}

The integer area of a triangle $\triangle ABC$ can be computed by the following formula.
\begin{equation*}
\is (\triangle ABC)=|\det(AB,BC)|.
\end{equation*}

Note that the integer area of a triangle is independent under relabelling, i.e., it does not depend on which sides are chosen to produce the sublattice. Additionally, the integer area of a polygon is independent on the decomposition into integer triangles.

An integer triangle is {\it empty} if its intersection with the lattice $\z^2$ consists only  of its vertices.
It turns out that the integer area of an empty triangle is always $1$.
On the other hand, the Euclidean area of every empty triangle is 1/2. Using these facts, it follows straightforwardly that the  integer area of any integer polygon is always twice the Euclidean area.
Additionally, if all the triangles in the decomposition are empty, then the integer area is equal to the number of such triangles.
We conclude this discussion with the following famous Pick's formula.

{\noindent
{\bf Pick’s formula.} Let $S$ be the Euclidean area of an integer polygon with $I$ integer points in its interior
and $E$ integer points on its boundary. Then the following relation holds
\begin{equation*}
S = I+\frac{E}{2}-1.
\end{equation*}
}

(For further discussions, see e.g. Chapter 2 of~\cite{my-book}.)

\subsubsection{Integer sine}
Using the notion of integer area, one can produce a natural notion of the \textit{integer sine function}.

\begin{definition}
Let $\angle ABC$ be a (non-trivial) rational angle with vertex $B$. Let $A'$ (\textit{resp.} $C'$) be the integer point of $BA$ (\textit{resp.} $BC$) which is closest to $B$.
Then the {\it integer sine}  of  $\angle ABC$ is the integer area of $\triangle{A'BC'}$. It is denoted as $\isin \angle{ABC}$.
The integer sine of a trivial angle is $0$.
\end{definition}

This definition immediately leads to the following statement.

\begin{proposition}\label{sine-rule-prop}
The integer sine satisfies
\begin{equation*}
 \isin(\angle ABC)=\frac{{\is (\triangle ABC)}}{{\il (AB)\cdot \il (BC)}}.
\end{equation*}
\end{proposition}

As a consequence of Proposition~\ref{sine-rule-prop}, we obtain an integer analogue of the Euclidean sine rule.

\begin{proposition}
For any rational angle $\angle ABC$, we have
\begin{equation*}
\frac{\isin \angle ABC}{\il (AC)}= \frac{\isin \angle BCA}{\il (AB)}=\frac{\isin \angle CAB}{\il (BC)}= \frac{{\is (\triangle ABC)}}{{\il (AB)\cdot \il (AC) \cdot \il (BC)}}.
\end{equation*}
\end{proposition}

\subsubsection{Sails, LLS sequences and integer tangents}
%\hphantom{}\newline
In order to define the integer tangent, let us first introduce the notion of a \textit{sail} of an angle.

\begin{definition}
Let $\alpha$ be a rational angle centred at a point $x$.
Then the {\it sail} of $\alpha$ is the boundary of the convex hull of all integer points inside the angle, excluding $x$.
\end{definition}

Note that the sail is a \textit{broken line}.

\begin{definition}[Lattice length sine sequence]
Let $\alpha$ be a rational angle and let $A_0A_1\ldots A_n$ be its sail.
Then the {\it lattice length sine sequence} of $\alpha$ (or {\it LLS sequence})  is the sequence  defined as follows
\begin{align*}
a_{2k} &=\il(A_kA_{k+1}),\\
a_{2k-1} &= \isin(A_{k-1}A_kA_{k+1}). \notag
\end{align*}
 We denote this LLS sequence as $LLS(\alpha)$.
\end{definition}

\begin{figure}[hbt]
\centering
\includegraphics[width=0.45\linewidth]{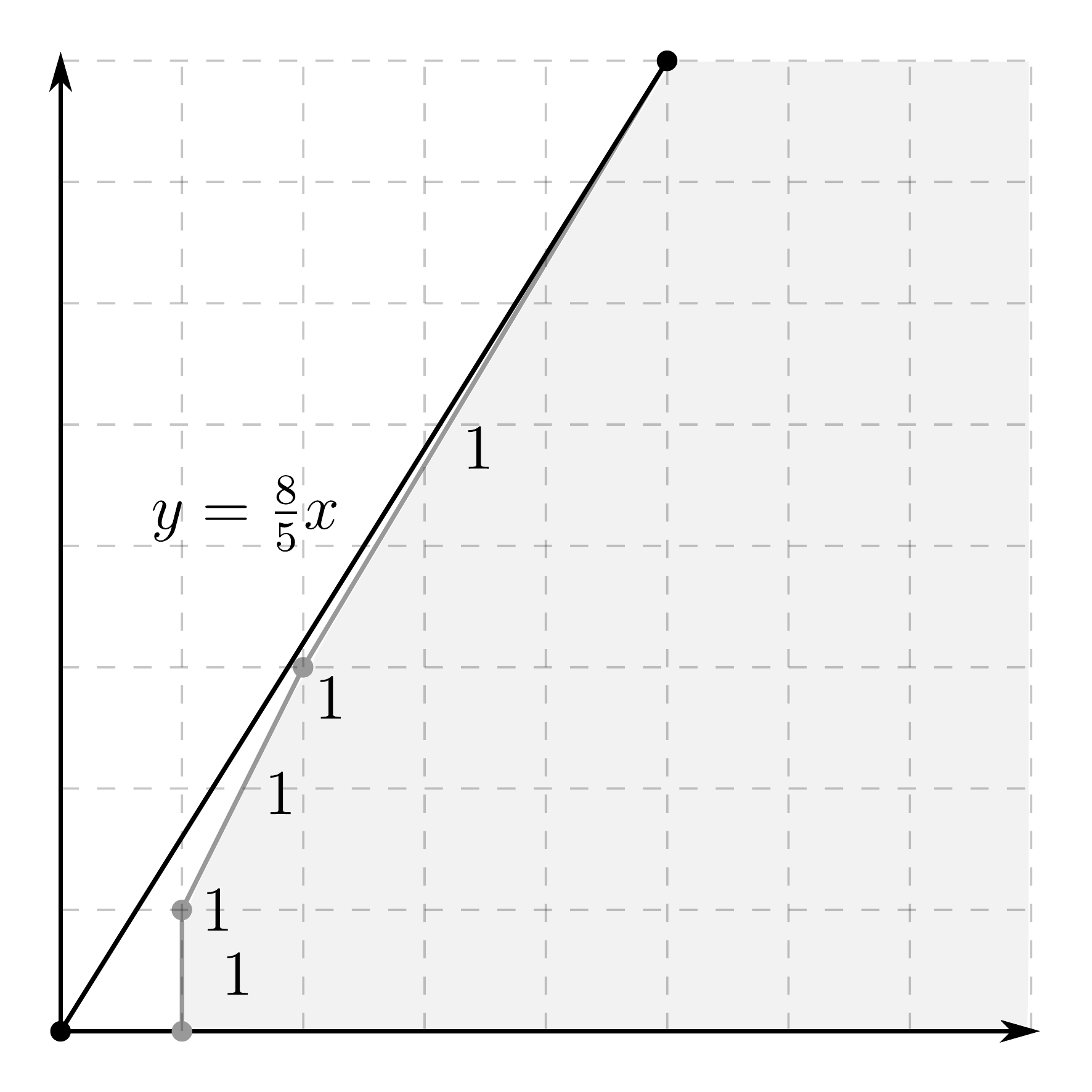}
\caption{An example of the sail of the angle formed by the rays $\left\{\left(x,\frac{8}{5}x\right)\big|x\geq{0}\right\}$ and $\left\{\left(x,0\right)|x\geq{0}\right\}$. See Example~\ref{LLS}. }\label{fig:LLS}
\end{figure}

\begin{example}\label{LLS}
Let $\alpha$ be the angle formed by the rays $\left\{\left(x,\frac{8}{5}x\right)\big| x\geq{0}\right\}$ and $\left\{\left(x,0\right)|x\geq{0}\right\}$. The sail is the broken line connecting the points $(0,1)$, $(1,1)$, $(2,3)$ and $(8,5)$.
The corresponding LLS sequence is $(1,1,1,1,1)$.
\end{example}

We now define the integer tangent.

\begin{definition}[Integer tangent]
Let $(a_0,\ldots, a_{2n})$ be the LLS sequence of a non-trivial angle $\alpha$.
Then the \textit{integer tangent} $\itan\alpha$ of $\alpha$ is the rational number which has the continued fraction expansion $[a_0;a_1:\ldots: a_{2n}]$, i.e.,
\begin{equation*}
\itan \alpha = a_0+\frac{1}{\displaystyle a_1+\frac{1}{\displaystyle\ddots+\frac{\displaystyle 1}{a_{2n}}}}.
\end{equation*}
The integer tangent of a trivial angle is 0.
\end{definition}

Recall that a continued fraction $[a_0;a_1:\ldots: a_n]$ is {\it regular} if all $a_i$ are positive integers, for $i > 0$ (in other words, only $a_0$ may be negative or zero).
If the number of elements in the continued fraction is odd (even) then the corresponding continued fraction is called
{\it odd} (respectively, {\it even}).
Note that every rational number has a unique regular
continued fraction expansion with an odd number of partial quotients and a unique continued fraction expansion with an even number of partial quotients.

\subsubsection{Integer cosines, integer arctangents and  Hermite normal forms}
Unlike integer sine and integer tangent, there is no straightforward geometric definition of an integer cosine.
Instead, we define it in terms of integer sine and integer tangent.

\begin{definition}[Integer cosine]
Let $\alpha$ be a (non-trivial) rational angle. Then the {\it integer cosine} is defined as
\begin{equation*}
\icos \alpha =\frac{\isin \alpha}{\itan \alpha}.
\end{equation*}
The integer cosine of a trivial angle is set to be $1$.
\end{definition}

Note that if $\alpha$ is not a trivial angle, then $\icos\alpha$ is a non-negative integer. Furthermore, if $\isin\alpha\neq{1}$, then $\icos\alpha$ is strictly smaller than $\isin{\alpha}$.

Let us now define the integer arctangent.
Like in the Euclidean case, the integer arctangent is only defined relative to a fixed coordinate system.

\begin{definition}[Integer arctangent]
Let $q=m/n\geq{1}$ be a rational number with relatively prime integers $m\ge n>0$. Then, the {\it integer arctangent}  of $q$
is the rational angle centred at the origin which has edges passing through the points $(1,0)$
and $(n,m)$. We denote it by $\iarctan q$.
\end{definition}

For every (non-trivial) integer angle there exists a unique integer arctangent that is congruent to this angle.
In particular, integer arctangents act as a type of normal form for the integer congruence classes of rational angles. By rewriting the edges of an integer arctangent as columns in a matrix, we can identify the integer arctangent with a matrix of the following form
\begin{equation*}
\left(
\begin{array}{cc}
1& \icos\alpha\\
0& \isin\alpha\\
\end{array}
\right)\quad
\end{equation*}
where $0<\icos \alpha < \isin \alpha$ $($excluding the case $\isin\alpha={1})$.
Matrices of this form are said to be {\it Hermite normal forms}.

\begin{remark}
When $\alpha=\iarctan1$ the Hermite normal form is the identity matrix, but $\isin\alpha=\icos\alpha=\itan\alpha=1$.
\end{remark}

\begin{remark}
Whilst integer sine and integer cosine are not defined for irrational angles, one can still define a sail and, therefore, the integer tangent. This integer tangent is invariant under $\text{Aff}(2,\z)$.
\end{remark}

\begin{remark}
Recall the Euclidean cosine rule for a triangle $\triangle ABC$
Let $|BC|=a$, $|AC|=b$, $|AB|=c$, and $\alpha=\angle BAC$, then we have
\begin{equation*}
\cos \alpha =\frac{b^2+c^2-a^2}{2bc}.
\end{equation*}
\end{remark}

A generalisation of this rule in integer trigonometry is currently unknown.

\begin{problem}[2008, \cite{itrig1}]
Find an integer analogue of the cosine rule.
\end{problem}

\subsection{Integer Trigonometry}
\subsubsection{Transpose and adjacent angles}

Let us start with the following definition.

\begin{definition}
Let $\alpha= \angle BAC$ be a rational angle. Then
\begin{itemize}
\item the angle $\angle CAB$ is the {\it transpose}, denoted as $\alpha^t$.
\item the angle $\angle CAD$ with $D=A-AC$ is the {\it adjacent angle}, denoted as $\pi - \alpha$.
\end{itemize}
\end{definition}

This leads to two remarkable properties for the trigonometric functions of transpose and adjacent angles
(see~\cite{itrig1, my-book} for the proofs).

\begin{proposition}\label{Prop:transpose}
The trigonometric functions of transpose angles satisfy the below equations
\begin{equation*}
\left\{
\begin{array}{l}
\isin \alpha^t=\isin \alpha\\
\icos \alpha^t \cdot \icos \alpha \equiv 1 \mod \isin\alpha
\end{array}
\right.
\end{equation*}
Furthermore, if $[a_0;a_1:\ldots:a_{2n}]$ is the regular odd continued fraction expansion of $\itan\alpha$, then
\begin{equation*}
\itan{\alpha^t}=[a_{2n};a_{2n-1}:\ldots:a_0].
\end{equation*}
\end{proposition}

\begin{proposition}\label{Prop:adjacent}
The trigonometric functions of adjacent angles satisfy the below equations.
\begin{equation*}
\left\{
\begin{array}{l}
\isin (\pi -\alpha)=\isin \alpha\\
\icos (\pi - \alpha) \cdot \icos \alpha \equiv -1 \mod \isin\alpha
\end{array}
\right.
\end{equation*}
Furthermore, if $[b_0;b_1:\ldots:b_{2n+1}]$ is the regular even continued fraction expansion of $\itan\alpha$, then
\begin{equation*}
\itan\pi-\alpha=[b_{2n+1};b_{2n}:\ldots:b_0].
\end{equation*}
\end{proposition}

\begin{remark}
To the authors' best knowledge, the second  statement of Proposition~\ref{Prop:adjacent} is new.
\end{remark}

\begin{example}\label{transpose}
Let $\angle{BAC}$ be a rational angle with LLS sequence $(1,1,1,1,1,2,1)$, see Figure~\ref{fig:transpose}. Then the LLS sequence of $\angle{CAB}$ is found by reversing the LLS sequence of $\angle{BAC}$, i.e., $(1,2,1,1,1,1,1)$. For both angles the integer sines are
$$
\isin{\angle{BAC}}=\isin{\angle{CAB}}=29,
$$
but the integer cosines are
$$
\icos{\angle{BAC}}=18 \quad  \hbox{and} \quad \icos{\angle{CAB}}=21.
$$
Note that $18\cdot{21}=378=29\cdot{13}+1\equiv 1\mod 29$.
\end{example}

\begin{figure}[hbt]
\centering
\begin{subfigure}[b]{0.45\textwidth}
            \centering
            \includegraphics[width=\textwidth]{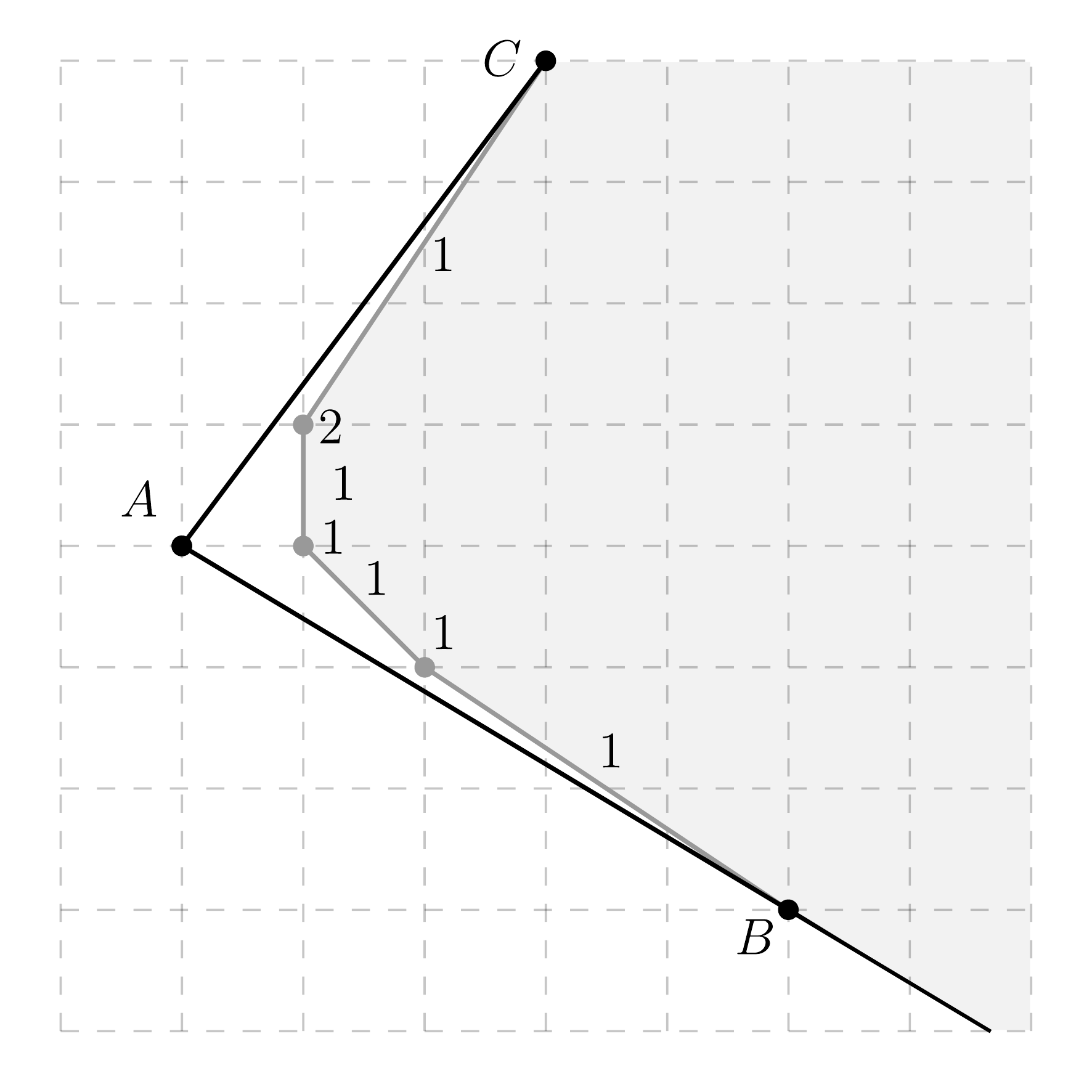}
\caption{An example of transpose angles, see Example~\ref{transpose}.}\label{fig:transpose}
        \end{subfigure}
        \quad
        \begin{subfigure}[b]{0.45\textwidth}
            \centering
            \includegraphics[width=\textwidth]{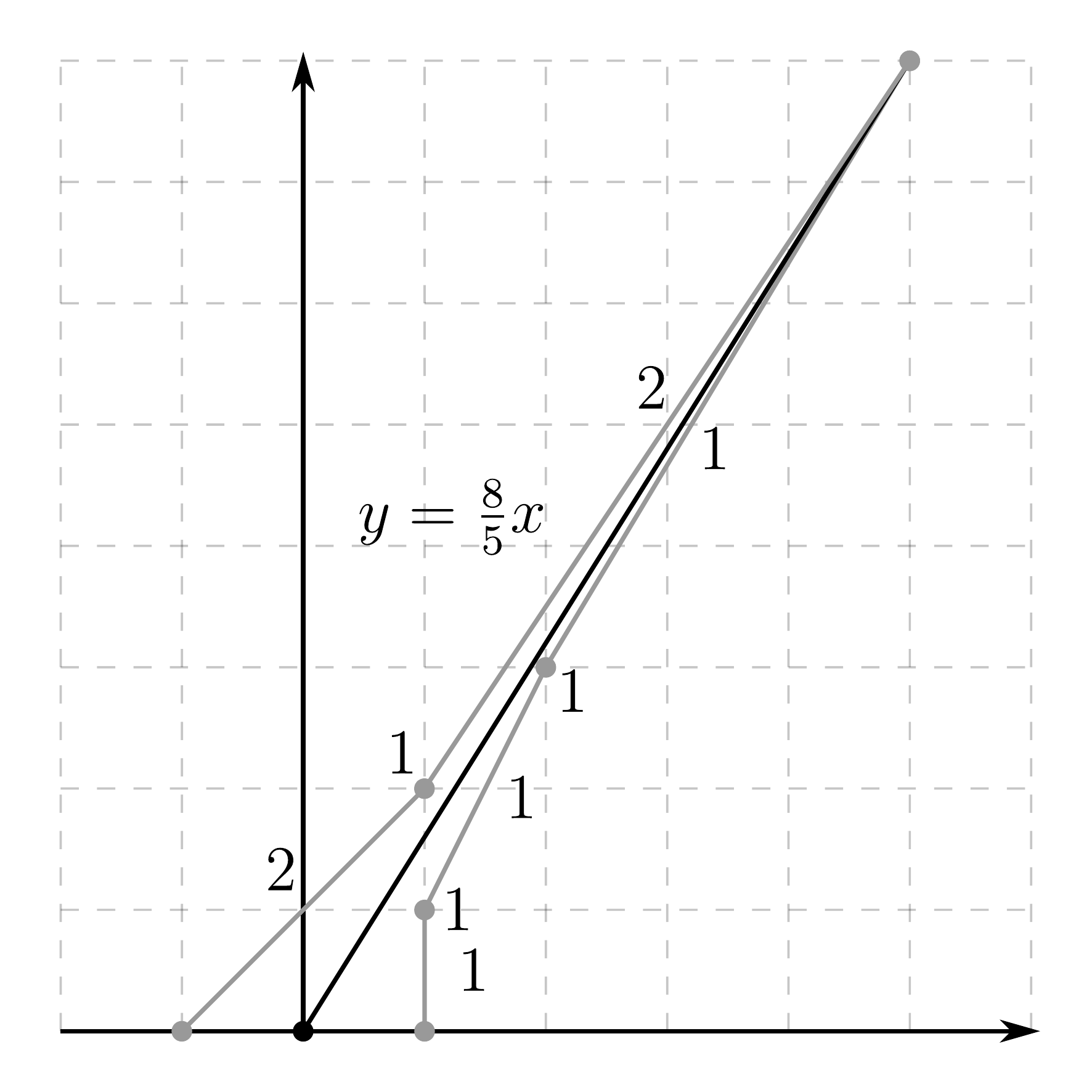}
\caption{An example of adjacent angles,  see Example~\ref{adjacent}.}\label{fig:adjacent}
        \end{subfigure}
\caption{Examples of transpose and adjacent angles.}
\end{figure}

\begin{example}\label{adjacent}
Let $\alpha$ be the angle between the rays defined by $\{(x,0):x\geq{0}\}$ and $\{(x,8/5x):x\geq{0}\}$, then an adjacent angle can be formed between the rays defined by
 $\{(x,8/5x):x\geq{0}\}$ and $\{(-x,0):x\geq{0}\}$, see Figure~\ref{fig:adjacent}. The LLS sequences are $(1,1,1,1,1)$ and $(2,1,2)$, respectively. The integer sine functions are
$$
\isin\alpha=\isin\pi-\alpha =8
$$
and the integer cosine functions are
$$
\icos\alpha=5 \quad \hbox{and} \quad  \icos\pi-\alpha=3.
$$
Note that $3\cdot 5=15=2\cdot 8-1\equiv-1\mod8$, $8/5=[1;1:1:2]$, and $8/3=[2;1:1:1]$.
\end{example}

\begin{remark}{\bf(On right integer angles).}
An integer angle is a {\it right angle} if it is integer congruent to both its adjacent and its transpose angles.
It turns out that up to integer congruency there are exactly two integer right angles in $\mathbb{R}^2$. They are $\iarctan 1$ and $\iarctan 2$.
\end{remark}

\subsubsection{Summation of angles}\label{Summation of angles}

One of the most surprising things in integer trigonometry is that angle summation is not uniquely defined up to integer conjugacy classes.
For example, by summing two angles that are integer congruent to $\iarctan 1$, we can obtain a straight line and an angle that is integer congruent to $\iarctan 1$:

\begin{center}
\includegraphics[scale=1]{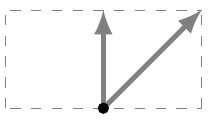}
\raisebox{0.5cm}{+}
\includegraphics[scale=1]{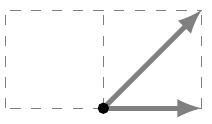}
\raisebox{0.5cm}{=}
\includegraphics[scale=1]{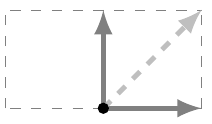}

\includegraphics[scale=1]{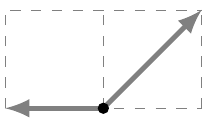}
\raisebox{0.5cm}{+}
\includegraphics[scale=1]{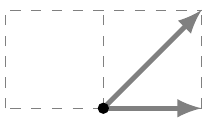}
\raisebox{0.5cm}{=}
\includegraphics[scale=1]{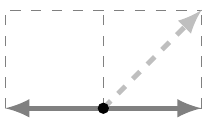}
\end{center}

(Note that all angles on the left-hand side are integer congruent to  $\iarctan 1$.)

In fact, the summation of angles is defined only up to an integer parameter. This leads to the following definition.
\begin{definition}
Let $(a_0,\ldots, a_{2n})$ and
$(b_0,\ldots, b_{2m})$
be the LLS-sequences for two integer angles $\alpha$ and $\beta$.
For each $s\in\z$, we set
$
\alpha+_s\beta
$
to be the angle summation of $\alpha$ and $\beta$ that has
the LLS sequence
$$
(a_0,\ldots, a_{2n},s,b_0,\ldots, b_{2m}).
$$
\end{definition}

\begin{figure}[hbt]
\centering
\begin{subfigure}[b]{0.4\textwidth}
            \centering
            \includegraphics[width=\textwidth]{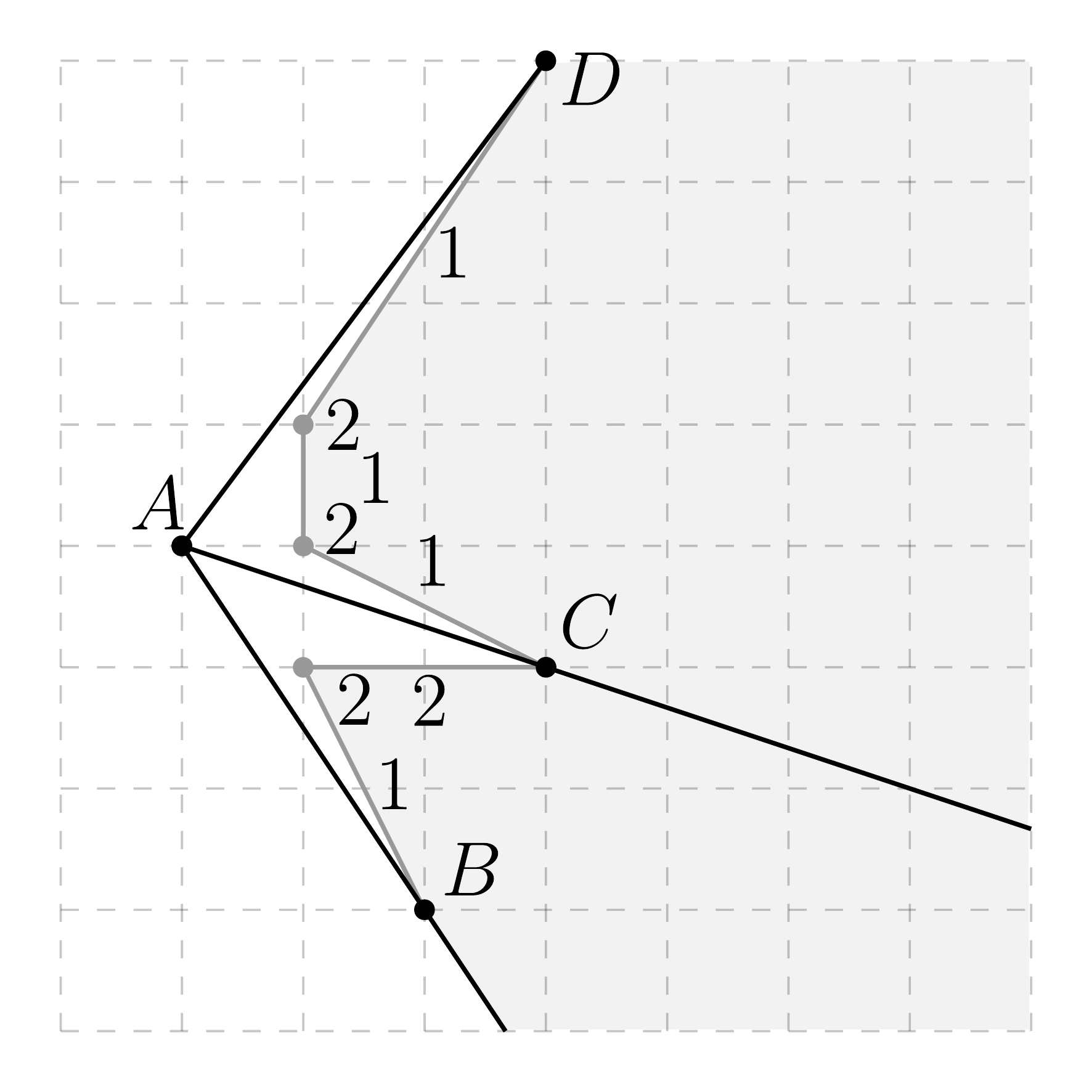}
            \caption{The LLS sequence of two angles: $\angle BAC$ and $\angle CAD$}
        \end{subfigure}
        \quad
        \begin{subfigure}[b]{0.4\textwidth}
            \centering
            \includegraphics[width=\textwidth]{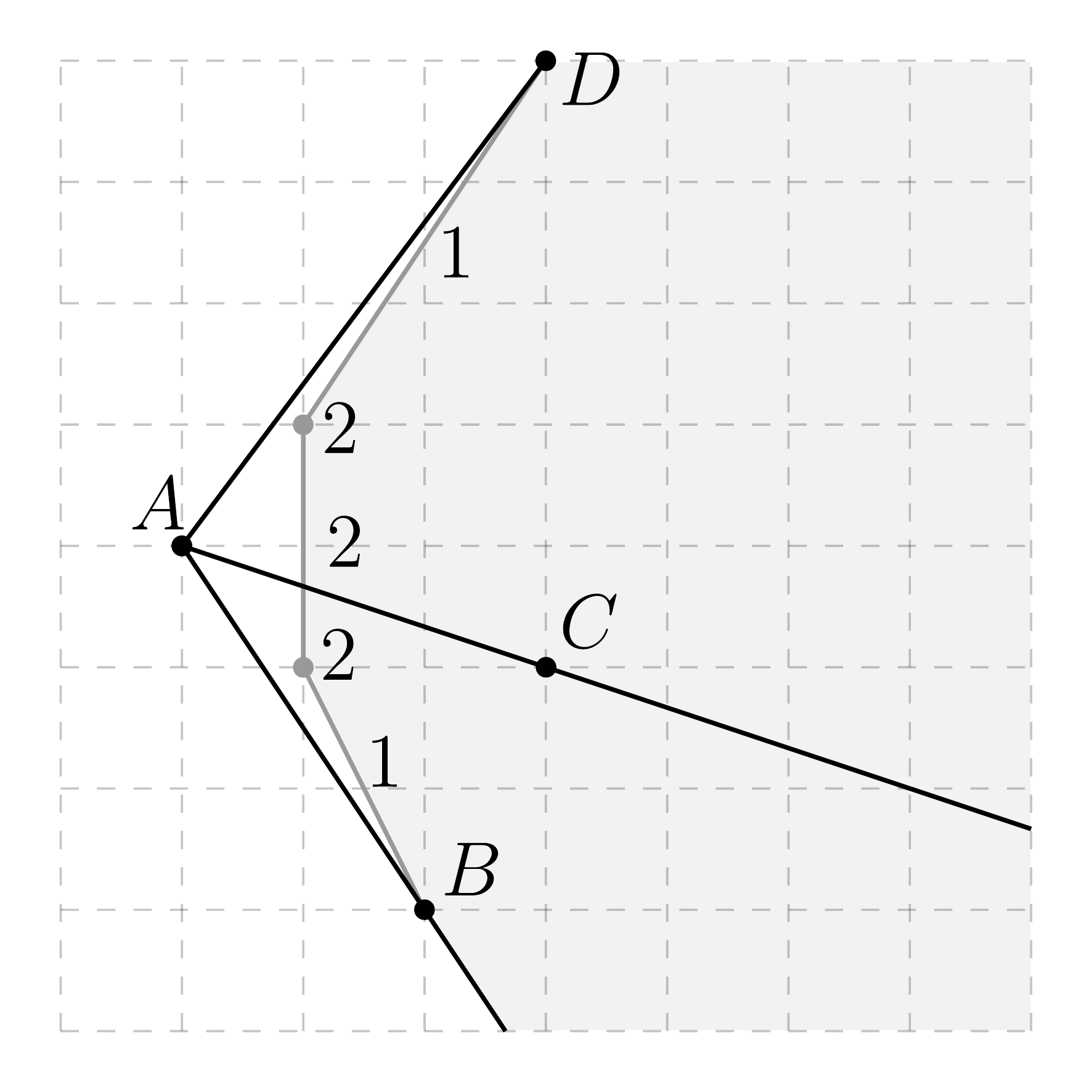}
             \caption{The LLS sequence of the angle sum $\angle BAD=\angle BAC+_{-1}\angle CAD$.}
        \end{subfigure}
\caption{An example of angle summation and the effect this has on the LLS sequences, see Example~\ref{anglesum}.}\label{fig:anglesum}
\end{figure}

\begin{example}\label{anglesum}
Let $\angle BAC$ be an angle with LLS sequence $(1,2,2)$ and let $\angle(CAD)$ be an angle with LLS sequence $(1,2,1,2,1)$, as in Figure~\ref{fig:anglesum}. Then, the LLS sequence of the combined angle is $(1,1,2,2,1)$. Note that \[\itan \angle{BAD}=\frac{17}{10}=[1;1,2,2,1]=[1;2,2,-1,1,2,1,2,1].\] Here $s=-1$.
\end{example}

\subsubsection{Angles in integer triangles}

In this section we discuss a criterium for three integer angles
to be the angles of an integer triangle.

In Euclidean geometry, the angles $\alpha$, $\beta$, and $\gamma$
are angles of some triangle if and only if
$$
\alpha+\beta +\gamma =\pi.
$$
As we have seen in Subsection~\ref{Summation of angles},
angle summation is not uniquely defined on integer conjugacy classes, and so there is not an exact generalisation of
this Euclidean condition in terms of integer geometry.
Instead, we can look at a similar condition on the tangent function.
\begin{proposition}
There exists a $($Euclidean$)$  triangle with angles $(\alpha, \beta, \gamma)$, where $\alpha$
is assumed to be acute, if and only if the following two conditions hold
\begin{equation*}
\left\{
\begin{array}{l}
\tan (\alpha+\beta+\gamma)=0\\
\tan (\alpha+\beta)\notin [0,\tan \alpha] \: .
\end{array}
\right.
\end{equation*}
\end{proposition}

In \cite{itrig1}, a generalisation was shown for integer triangles (Proposition~\ref{triangleCriteria}). Let us start with the following notation.

For a sequence of rational numbers $q_1,\ldots,q_k$
with odd regular continued fractions $q_i=[a_{i,0};a_{i,1}:\ldots : a_{i,2n_i}]$ for $i=1,2\ldots, k$,
we set
$$
]q_1:q_2:\ldots:q_{k}[=
[a_{1,0};a_{1,1}:\ldots: a_{1,2n_1}
:a_{2,0}:a_{2,1}:\ldots: a_{2,2n_2}:\ldots:
a_{k,0}:a_{k,1}:\ldots: a_{i,2n_k}].
$$

\begin{proposition}\label{triangleCriteria}
There exists an integer triangle with three given angles
if and only if there exists an ordering $(\alpha,\beta,\gamma)$ satisfying the following two conditions
\begin{equation*}
\left\{
\begin{array}{l}
 ]\itan \alpha:-1:\itan \beta:-1:\itan \gamma[ =0\\
 ]\itan \alpha: -1: \itan\beta[ \notin [0,\itan \alpha]
\end{array}
\right.
.
\end{equation*}
\end{proposition}

\begin{figure}[hbt]
\centering
\includegraphics[width=0.5\linewidth]{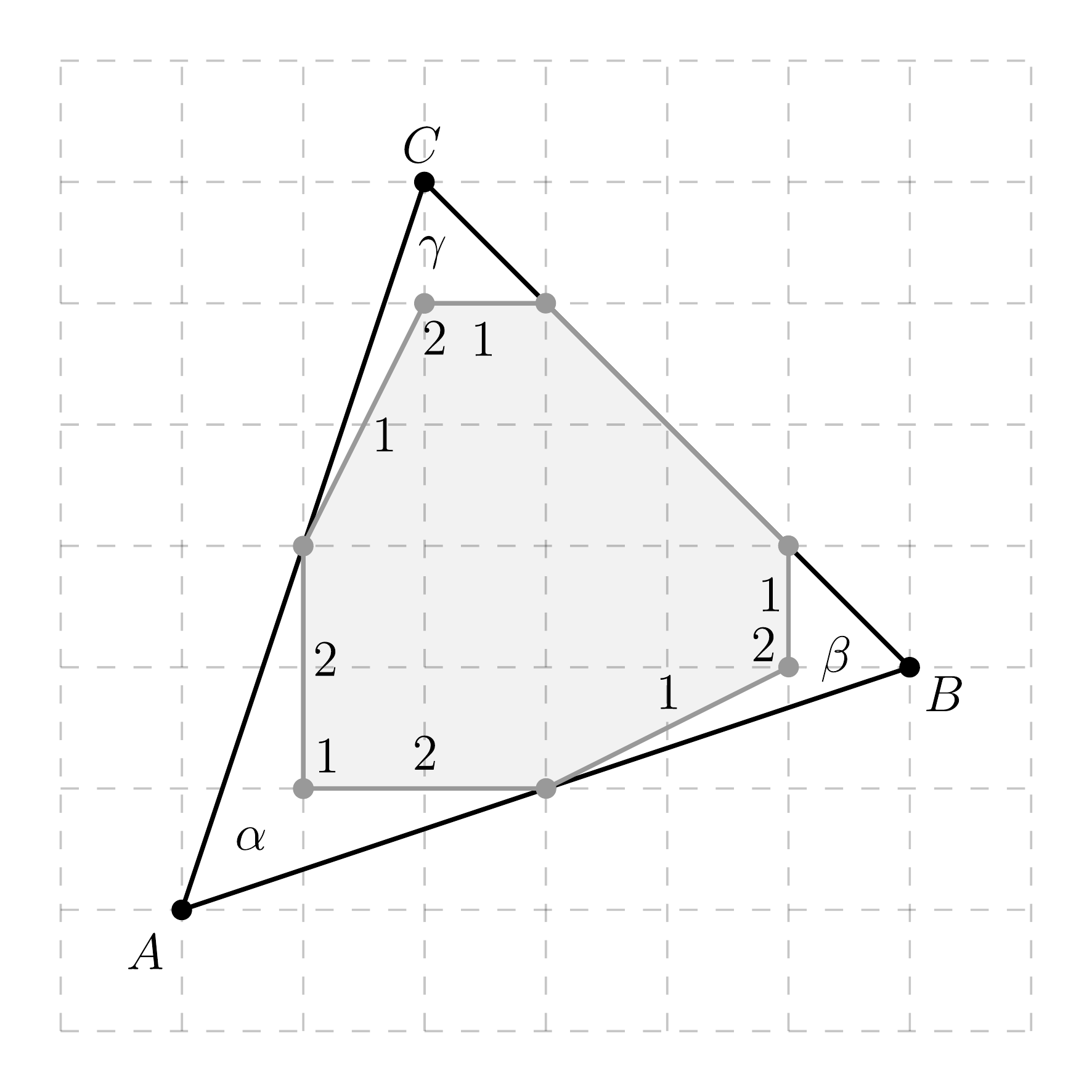}
\caption{An example of the conditions imposed on the LLS sequences of  the angles inside an integer triangle, see Example~\ref{triangle}.}\label{fig:triangle}
\end{figure}
\begin{example}\label{triangle}
Let $\alpha$, $\beta$ and $\gamma$ be rational angles with LLS sequences $(2,1,2)$, $(1,2,1)$ and $(1,2,1)$, respectively. Then, we have
\begin{align*}
& [2;1:2:-1:1:2:1:-1:1:2:1]=0, \\
& [2;1:2:-1:1:2:1]=4>\frac{8}{3}=[2;1:2].
\end{align*}

Therefore, Proposition~\ref{triangleCriteria} implies that there is an integer triangle which has $\alpha$, $\beta$ and $\gamma$ as its angles, see Figure~\ref{fig:triangle}.
\end{example}

This leads to the following open problem.

\begin{problem}\label{IKEA-problem-2d}{\bf (IKEA problem.)}
Find a necessary and sufficient condition for a collection of rational angles
$(\alpha_1,\ldots,\alpha_k)$ to be the angles of some integer $k$-gon.
\end{problem}

Currently it is only known that the angles $\alpha_1,\ldots,\alpha_k$ of a $k$-gon are required to satisfy
the following relation
$$
]\itan\alpha_1:m_1: \itan\alpha_2:m_2:\ldots:m_{k-1}: \itan\alpha_k[=0,
$$
for some choice of integers $m_i\ge -1$ (see~\cite{my-book} for further discussion).

\section{Definitions of Integer Trigonometric Functions in Higher Dimensions}\label{Sec2}

\subsection{Integer volume and integer sine}

We start with the definitions of integer areas and integer sines.

\begin{definition}
Consider an integer $k$-simplex $S=A_0A_1\ldots A_{k}$ in $\r^n$.
Let $\Gamma$ be the lattice of all integer vectors in the $k$-plane spanned by this simplex
and let $\Gamma_S$ be the sublattice generated the edges, i.e., by $A_iA_j$ for all $0\leq{}i<j\leq{}k$.
The \textit{integer volume} of $S$ is the index $\iv(S)=[\Gamma:\Gamma_S]$.
\end{definition}

\begin{definition}
Consider a rational simplicial $k$-cone $\alpha$ in $\r^n$.
Let $\Gamma$ be the lattice of all integer vectors in the  $k$-plane spanned by $\alpha$
and let $\Gamma_\alpha$ be the sublattice of all integer vectors lying on the edges of $\alpha$.
The \textit{integer sine} of $\alpha$ is the index $\isin(\alpha)=[\Gamma:\Gamma_{\alpha}]$.
\end{definition}

With these definitions in place, it is easy to check the following formula holds.

\begin{proposition}
Let $S=A_0A_1\ldots A_{k}$ be a $k$-simplex.
Then, the following equation holds
\begin{equation*}
\isin(\angle(A_0;A_1\ldots A_k) =\frac{\iv(S)}{\il(A_0A_1)\cdots{\il(A_0A_{k})}}.
\end{equation*}
\end{proposition}

\subsection{Arctangents of rational cones}

We define the integer arctangents in a similar way to how we defined them in the 2-dimensional case.

\begin{definition}\label{arctan-def}
Consider an ordered integer $k$-cone $\alpha$
whose edges are defined by vectors $v_1,\ldots, v_k$.
Fix some integer coordinate system.
We say that $\alpha$ is an {\it integer arctangent}
in this coordinate system, if the coordinates of each edge $v_i$ satisfy
$$
v_i=(x_{1,i},\ldots x_{i,i}, 0,\ldots, 0), \quad i=1,\ldots,k,
$$
such that $0\le x_{j,i} < x_{i,i}$ for all $0<j<i$.
\end{definition}

\begin{remark}
Note that here we consider all integer cones (the corresponding edges may be non-integer).
Integer arctangents are complete invariants of ordered cones.
\end{remark}

Recall that the $n\times  k$-matrix $M$ is called the {\it Hermite normal form} if it is of the form
$$
M=\left(
\begin{array}{cccc}
 a_{1,1}& a_{1,2}&\ldots  &a_{1,k}\\
0& a_{2,2}&\ldots  &a_{2,k}\\
 & 0&\ddots  & \vdots\\
 \vdots& &  &a_{k,k}\\
  & \vdots & \ddots  &0\\
 & & & \vdots \\
0& 0 &\ldots  &0\\
\end{array}
\right)
$$
where, for all $0<j<i$, we have
$$
0\le a_{j,i}<a_{i,i}.
$$

\begin{definition}
We say that a Hermite normal form is {\it normalised} if the integer length of each column vector is 1, i.e.,
the greater common divisor of the column is 1.
\end{definition}

\begin{remark}
It is a classical result that the Hermite normal form uniquely characterises the integer conjugacy classes of the matrices associated to these cones, see~\cite{Schrijver1998}.
It is clear that normalised Hermite normal forms are in one-to-one correspondence with integer arctangents.
Therefore, each rational cone has a unique normalised Hermite normal form corresponding to its arctangent.
\end{remark}

\begin{remark}
Given a rational $k$-cone in $\r^n$, there are a number of algorithms one can use to find the corresponding normalised Hermite normal form. These algorithms are typically based on subtractive algorithms, which act as a generalisation of the Euclidean algorithm. The running time of such algorithms is usually polynomial, e.g., see \cite{KannanBachem:1979}.
\end{remark}

\begin{definition}
Given a rational cone $\alpha$ the \textit{associated Hermite normal form} is the normalised matrix whose columns coincide with the ordered edges of $\iarctan\alpha$ (in the appropriate coordinate system).
The upper $k\times k$ submatrix of this matrix is denoted as $\IARCTAN(\alpha)$.
\end{definition}

\begin{proposition} The following statements hold.
\begin{itemize}
    \item The integer arctangent is a complete invariant of rational cones.
    \item The normalised Hermite normal form is a complete invariant of rational cones.
\end{itemize}
\end{proposition}

We use these normalised Hermite normal forms to obtain the integer trigonometric functions for rational cones.

\begin{definition}
Let $N=\IARCTAN(\alpha)$.
Then the element $a_{i,i}$ is said to be the \textit{$i$-th integer sine} and denoted by $\isin_i\alpha$.
The element $a_{j,i}$ with $j<i$ is the \textit{$(j,i)$-th integer cosine} and denoted by $\icos_{j,i}\alpha$.
The \textit{$i$-th integer tangent} is the projectivisation of the $i$-th column vector, denoted $\itan_i\alpha$.
\end{definition}

Here by projectivisation, we simply mean that we consider the vector up to non-zero scalar multiplication.

\begin{example}
In the two-dimensional case the normalised Hermite normal form for $\alpha$ (when $\alpha\neq{\iarctan{1}}$) is written as follows
\begin{equation*}
\left(
\begin{array}{cc}
1& \icos\alpha\\
0& \isin\alpha\\
\end{array}
\right)\quad
\end{equation*}

Here the integer tangent is the following projective vector $\big(\icos\alpha{:}\isin\alpha\big)$; it is naturally identified with the fraction $\frac{\isin\alpha}{\icos\alpha}$.
\end{example}

\begin{example}
Let $\alpha$ be a rational cone with generating vectors
$$
v_1=
\left(\begin{array}{c}
13\\8\\4
\end{array}\right),
\quad
v_2=
\left(\begin{array}{c} 7\\-3\\11\end{array}\right),
\quad
v_3= \left(\begin{array}{c} -19\\16\\-5\end{array}\right).
$$
Then
$$
\IARCTAN(\alpha)=
\left(\begin{array}{ccc} 1& 4& 67\\ 0& 5& 59\\0&0&107
\end{array} \right).
$$
In particular,
$$
\itan_1 \alpha=(4:5),
\quad \hbox{and} \quad
\itan_2 \alpha=(67:59:107),
$$
and
$$
\icos_{1,2}\alpha=5, \quad \isin_1\alpha =5, \quad
\icos_{1,3}\alpha=67, \quad \icos_{2,3}\alpha =59 \quad \hbox{and} \quad
\quad \isin_2\alpha =107.
$$
\end{example}

We conclude this section with the following general question.

\begin{problem}
Extend the notions of integer trigonometry to irrational $k$-cones.
\end{problem}

\section{Multidimensional Integer Trigonometry}\label{Sec3}

\subsection{Simple rational cones}

For simplicity, in this section we will entirely work with \textit{(ordered) simple rational cones}.

\begin{definition}
We say that a rational cone $\alpha$ is {\it simple} if any of its $(k-1)$-subcones, say $\beta$, satisfy $\isin(\beta)=1$.
\end{definition}

The simple rational cones have integer arctangents of the form
$$
\IARCTAN (\alpha)=
\left(
\begin{array}{ccccc}
1& 0&\ldots  &0&\icos_{1,k} \alpha\\
0& 1&\ldots  &0&\icos_{2,k} \alpha\\
\vdots& \vdots&\ddots  & \vdots & \vdots\\
0& 0&\ldots  &1&\icos_{k-1,k} \alpha\\
0& 0&\ldots  &0&\isin_{k} \alpha\\
\end{array}
\right).
$$
Note here that the converse is not true: the fact that the integer arctangent is in the above form does not imply that the cone is simple.

Discussion regarding non-simple rational cones can be found in Section~\ref{A few words on general lattices}.

\subsection{Integer trigonometric functions and transpositions of cones}

For $k$-cones with $k\geq{3}$, there is not a single unique way to take a transpose of a cone. Instead a transposition corresponds to a permutation of the edges.

\begin{definition}
Let $\alpha$ be a $k$-cone and let $s$ be a permutation in $S_k$.
Let $\alpha_s$ denote the cone obtained from $\alpha$ by permuting the (ordered) edges of $\alpha$ by $s$.
Then the angle $\alpha_s$ is {\it $s$-transpose} of $\alpha$.
\end{definition}

For simplicity, we write permutations in canonical cyclic notation.

The determinant of a cone (and therefore, the integer sine) does not depend on the order of the edges, leading to following statement.

\begin{proposition}\label{transpose-highdim-sin}
For a simple rational angle $\alpha$ and any transposition $s\in{S_k}$ we have
\begin{equation*}
\isin_k\alpha=\isin_k\alpha_s.
\end{equation*}
\end{proposition}

On the other hand, the integer cosines have more complicated relations.

\begin{proposition}\label{transpose-highdim}
Consider a simple rational cone $\alpha$ and a transposition $\sigma=(i,j)$ for $i<j<k$. Then, we have
\begin{equation*}
\icos_{x,k} \alpha_{\sigma}
=
\icos_{\sigma(x),k} \alpha \quad \hbox{for $x\in\{1,2,\ldots,k-1\}$.}
\end{equation*}
If $j=k$, we have
\begin{equation*}
\icos_{x,k}\alpha_{\sigma}=
\left\{
\begin{array}{lll}
-(\icos_{i,k}\alpha)^{-1} &\mod \isin_k(\alpha) \quad&\hbox{if $x=i$;}\\
-\icos_{x,k}\alpha\cdot(\icos_{i,k}\alpha)^{-1} &\mod \isin_k(\alpha) &\hbox{otherwise.}\\
\end{array}
\right.
\end{equation*}
\end{proposition}
We prove this proposition in Subsection~\ref{Transpose angles -- proof}.

\begin{remark}
Note that given rational 2-cones (angles), the statements of Proposition~\ref{transpose-highdim-sin} and Proposition~\ref{transpose-highdim}
coincide with Proposition~\ref{Prop:transpose}.
\end{remark}

\begin{example}\label{ex1}
Let $\alpha$ be a cone with corresponding matrix
$$
M_\alpha=
\left(
\begin{array}{ccc}
123&13&19\\
234&-347&156\\
655&341&-456\\
\end{array}
\right)
.
$$
Then, we have
$$
\IARCTAN (\alpha) = \left(
\begin{array}{ccc}
1&0&9719300\\
0&1&8781600\\
0&0&21469421
\end{array}
\right)
\hbox{ and }
\IARCTAN (\alpha_{(1,3)})=
\left(
\begin{array}{ccc}
1&0&11154342\\
0&1&18378882\\
0&0&21469421
\end{array}
\right)
.
$$
Furthermore
$$
\begin{array}{cccc}
9719300\cdot 11154342 &\equiv& 1 &\mod 21469421;\\
8781600\cdot 11154342& \equiv& -18378882&\mod 21469421.\\
\end{array}
$$

\end{example}

\begin{remark}
Note that the formulae of Propositons~\ref{transpose-highdim-sin} and~\ref{transpose-highdim}
uniquely determine the trigonometric functions of the $s$-transpose cones,
since any permutation is a composition of transpositions, i.e., cycles of length 2.
\end{remark}

As a consequence, we obtain the following surprising relation regarding integer cosines.

\begin{corollary}\label{strictcycle}
Let $\alpha$ be a simple rational cone and let $\tau=(1,2,\ldots,k)$.
Then, for every $j\in \{1,2,\ldots, k\}$, we have
\begin{equation*}
\icos_{x,k}\alpha_{\tau^j}\equiv\left\{\begin{array}{lll} -\icos_{j,k}^{-1} &\mod \isin_{k}\alpha&\hbox{if $x=j$,}\\
-\icos_{j,k}^{-1}\cdot\icos_{x,k} &\mod \isin_{k}\alpha&\hbox{otherwise.}
\end{array}
\right.
\end{equation*}
\end{corollary}

In fact, a more general condition on $k$-cycles can also be deduced.

\begin{corollary}\label{cycle}
Let $\alpha$ be a simple rational cone and let $\tau=(i_1,i_2,\ldots,i_k)$ be a cycle of length $k$.
Then, for every $j\in \{1,2\ldots, k-1\}$, we have
\begin{equation*}
\prod\limits_{i=1}^k
\icos_{j,k}\alpha_{\tau^i}\equiv (-1)^k \mod \isin_{k}\alpha.
\end{equation*}
\end{corollary}

\begin{example}
Let $\alpha$ be the same cone as in Example~\ref{ex1} and let $\tau=(1,2,3)$.
Then
$$
\IARCTAN (\alpha) = \left(
\begin{array}{ccc}
1&0&9719300\\
0&1&8781600\\
0&0&21469421
\end{array}
\right),
\quad
\IARCTAN (\alpha_{\tau})=
\left(
\begin{array}{ccc}
1&0&18378882\\
0&1&11154342\\
0&0&21469421
\end{array}
\right)
\quad\hbox{and}
$$
$$
\IARCTAN (\alpha_{\tau^2})=
\left(
\begin{array}{ccc}
1&0&20652409\\
0&1&18802856\\
0&0&21469421
\end{array}
\right)
.
$$

Furthermore
$$
\begin{array}{ccc}
9719300\cdot 18378882\cdot 20652409 &\equiv& -1 \mod 21469421;\\
8781600\cdot 11154342 \cdot 18802856 &\equiv& -1\mod 21469421.\\
\end{array}
$$
\end{example}

Corollary~\ref{strictcycle} leads to the following surprising observation regarding the determinant of certain matrices.

\begin{proposition}\label{Determinant_of_Special_Matrix}
Let $\alpha$ be any simple rational $k$-cone, let $\tau=(1,2,\ldots,{k})$ and define the matrix $M_\alpha$ as follows
\begin{equation*}
M_\alpha=\left(
\begin{array}{ccccc}
0& \icos_{1,k} \alpha &  \icos_{2,k}\alpha&\ldots & \icos_{k-1,k}\alpha \\
\icos_{k-1,k} \alpha_{\tau}&  0  & \icos_{1,k} \alpha_{\tau}&\ldots&\icos_{k-2,k} \alpha_\tau  \\
\icos_{k-2,k} \alpha_{\tau^2}&  \icos_{k-1,k} \alpha_{\tau^2}  &0&\ddots&\vdots  \\
\vdots&\vdots& \ddots &\ddots &\icos_{1,k} {\alpha_{\tau^{k-2}}}\\
\icos_{1,k} \alpha_{\tau^{k-1}} &\icos_{2,k} \alpha_{\tau^{k-1}} & \ldots &\icos_{k-1,k} \alpha_{\tau^{k-1}}& 0 \\
\end{array}
\right).
\end{equation*}
Then
\begin{equation*}
\det(M_\alpha)=1-k\mod\isin\alpha
\end{equation*}
\end{proposition}

We prove Proposition~\ref{Determinant_of_Special_Matrix} in Subsection~\ref{DSM}.

\subsection{Integer trigonometric functions and adjacent  cones}

As is the case for transpose cones, if $k>2$ there are a number of ways of constructing adjacent cones.

\begin{definition}
Let $\alpha$ be the cone generated by $(v_1,\ldots, v_k)$. Then the \textit{$i$-th adjacent cone}, denoted $\pi_i-\alpha$, is the cone generated by $(v_1,\ldots,v_{i-1},-v_i,v_{i+1},\ldots,v_k)$.
\end{definition}

Similar to the case of transpose angles, we have the following statement for integer sines.

\begin{proposition}\label{adjacent-highdim-sin}
Let $\alpha$ be a simple rational cone.
Then, for all $i\in\{1,2,\ldots,k\}$, we have
\begin{equation*}
\isin_{k}\alpha=\isin_k(\pi_i-\alpha).
\end{equation*}
\end{proposition}

Again, the case for integer cosines is more complicated.

\begin{proposition}\label{adjacent-highdim}
Consider a simple rational cone $\alpha$.
If $i\in\{1,\ldots,k-1\}$, we have
\begin{equation*}
\icos_{x,k}(\pi_i-\alpha) =
\left\{
\begin{array}{ll}
\isin_{k} \alpha-\icos_{x,k}\alpha,& \hbox{if $x=i$;}\\
\icos_{x,k}\alpha,& \hbox{otherwise.}
\end{array}
\right.
\end{equation*}
Finally, if $i= k$, we have
\begin{equation*}
\icos_{x,k}(\pi_k-\alpha) =
\isin_{k} \alpha-\icos_{x,k}\alpha
\end{equation*}
for all  $x\in\{1,2,\ldots,k-1\}$.
\end{proposition}
We prove Proposition~\ref{adjacent-highdim} in Subsection~\ref{Adjacent angles -- proof}.

\begin{remark}
Note that for the case of integer angles the statement of Proposition~\ref{adjacent-highdim}
coincides with Proposition~\ref{Prop:adjacent}.
\end{remark}

\begin{example}\label{ex3}
Let $\alpha$ be the cone in Example~\ref{ex1}.
Then we have
$$
\IARCTAN (\alpha) = \left(
\begin{array}{ccc}
1&0&9719300\\
0&1&8781600\\
0&0&21469421
\end{array}
\right),
\quad
\IARCTAN (\pi_1-\alpha)=
\left(
\begin{array}{ccc}
1&0&11750121\\
0&1&8781600\\
0&0&21469421
\end{array}
\right)
.
$$
We check that
$$
9719300+11750121= 21469421.
$$
\end{example}

We conclude this subsection with a problem on right angled cones.

\begin{definition}
An integer cone is {\it right angled}
if it is integer congruent to all its adjacent and transpose cones.
\end{definition}

\begin{problem}
What are the integer right angled cones of dimension greater than $2$?
\end{problem}

\subsection{\texorpdfstring{Cones in $k$-simplices}{Cones in k-simplices}}

Surprisingly, there are some simple relations between cones which generate the same simplex.

\begin{proposition}\label{simplex}
Consider a $k$-simplex $A_0A_1\ldots A_k$ with unit integer lengths of all its edges.
Let
\begin{align*}
 \alpha=\angle(A_0;A_1,A_2,\ldots,A_k), \\
 \beta=\angle(A_1;A_0,A_2,\ldots,A_k)
\end{align*}
and assume that $\alpha$ and $\beta$ are simple.
Then
\begin{equation*}
\left\{
\begin{array}{l}
\isin_{k}\beta=\isin_{k}\alpha;\\
\icos_{1,k}\beta+\sum\limits_{i=1}^{k-1} \icos_{i,k}\alpha \equiv 1  \mod \isin_{k}\alpha;\\
\icos_{x,k} \beta =\icos_{x,k}\alpha, \quad \hbox{\text{for} $x\in\{2,\ldots, n-1\}$}. \\
\end{array}
\right.
\end{equation*}
\end{proposition}

We prove this proposition in Subection~\ref{Simplices -- proof}.

\begin{example}\label{ex4}
Consider the 3-simplex whose coordinates are
$$
A_0=(0,0,0), \quad
A_1=(123, 234,655), \quad
A_2=(13,-347,156),\quad
A_3=(19,156,-457).
$$
Then for the cones $\alpha$ and $\beta$ as in Proposition~\ref{simplex} we have
$$
\IARCTAN (\alpha) = \left(
\begin{array}{ccc}
1&0&9719300\\
0&1&8781600\\
0&0&21469421
\end{array}
\right),
\quad
\IARCTAN (\beta)=
\left(
\begin{array}{ccc}
1&0&2968522\\
0&1&8781600\\
0&0&21469421
\end{array}
\right).
$$
Furthermore
$$
9719300+(2968522+8781600)\equiv 1 \mod 21469421.
$$
(In fact the cone $\alpha$ in this example coincides with the cone $\alpha$ in Example~\ref{ex1}.)
\end{example}

We conclude this subsection with the following two problems.

\begin{problem}
Develop a $k$-dimensional version of the relation shown in Proposition~\ref{triangleCriteria}. Namely,
find necessary and sufficient conditions to determine whether any rational $k$-simplex can be constructed using a given collection of rational $k$-cones.
\end{problem}
\begin{problem}{\bf (Multidimensional IKEA problem.)}
Find necessary and sufficient conditions to determine whether any rational polyhedra of dimension $k$ can be constructed using a given collection of rational $k$-cones.
\end{problem}

\subsection{Euclidean algorithm for lattice angles and cones}

Let us introduce a natural operation on the integer congruence classes of simple cones.
We will use it later for constructing strong best approximations of cones.

\begin{definition}
Let $\alpha$ be a simple cone and let
$$
\IARCTAN (\alpha)=
\left(
\begin{array}{cccc}
1& \ldots  &0&a_1\\
 \vdots&\ddots  & \vdots & \vdots\\
0&\ldots  &1&a_{k-1}\\
 0&\ldots  &0&a_k\\
\end{array}
\right).
$$
Let $\sigma=(i,k)\in S_k$ be a transposition for some $i\in\{1,\ldots, k-1\}$.
The {\it $i$-th Euclidean reduction} of $\alpha$, denoted $E_i(\alpha)$, is the cone corresponding to the following matrix
$$
\left(
\begin{array}{cccc}
1& \ldots  &0&a_{\sigma(1)}\\
 \vdots&\ddots  & \vdots & \vdots\\
0&\ldots  &1&a_{\sigma(k-1)}\\
 0&\ldots  &0&a_i\\
\end{array}
\right).
$$
\\
We say that $\lfloor a_k/a_i \rfloor$ is the {\it $i$-th partial quotient} and denote it as $\chi_{i}(\alpha)$. Here, $\lfloor\cdot\rfloor$ is the classical \textit{floor function}.
\end{definition}

\begin{remark}
In the two-dimensional case, take $\sigma=(1,2)$. Then, we have
$$
\IARCTAN(E_{1}(\alpha))
=
\left(
\begin{array}{cc}
1& a_2-\lfloor a_2/a_1\rfloor a_1\\
0& a_1\\
\end{array}
\right).
$$
This corresponds to the usual reduction step in the classical Euclidean algorithm.
\end{remark}

\begin{remark}
In higher dimension, similarly we have
$$
\IARCTAN(E_{1}(\alpha))
=
\left(
\begin{array}{cccc}
1& \ldots  &0& a_{\sigma(1)}-\lfloor a_{\sigma(1)}/a_i\rfloor a_i\\
 \vdots&\ddots  & \vdots & \vdots\\
0&\ldots  &1&a_{\sigma(n-1)}-\lfloor a_{\sigma(1)}/a_{n-1}\rfloor a_i\\
 0&\ldots  &0&a_i\\
\end{array}
\right).
$$
This corresponds to one of the multidimensional subtractive algorithms.
For further information on multidimensional subtractive
algorithms we refer to~\cite{Schweiger2000}.
\end{remark}

\subsection{Strong best approximation for integer cones}

\subsubsection{The two-dimensional case}
We say that a rational number $p/q$ with $q>0$
is a {\it strong best approximation of a real number} $\alpha$
if for every fraction $p'/q'$ with $0<q'<q$ we have
\begin{equation}\label{strong-best-classic}
|q'\alpha-p'|>|q\alpha-p|.
\end{equation}

\begin{remark}
This concept of a strong best approximation is also referred to as a \textit{best approximation of the second kind}, see \cite{Khin:1963}.
\end{remark}

Let us now formulate a classical fundamental property of strong best approximations,
which allows us  to compute all strong best approximations.

\begin{theorem}
Let $\alpha$ be a real number $($$\alpha\neq{a_0+1/2}$$)$ with the regular continued fraction expansion $[a_0;a_1:\ldots:a_m]$, where $a_0$ is an integer, $a_i$ is a positive integer for $i\ge 1$, and $a_m>1$ $($assuming $m\neq{0})$.
Then, the set of strong best approximations is
$$
\big\{[a_0;a_1:\ldots:a_i]\big| i=0,1,\ldots,m\big \}.
$$
\end{theorem}

For the proof of this statement see~\cite[Theorems~16 and 17]{Khin:1963}.
\begin{remark}
If $\alpha=a_0+1/2$, then $a_0$ is not a strong best approximation because
$$|1\cdot\alpha-(a_0+1)|=|1\cdot\alpha-(a_0)|.$$
\end{remark}

Let us introduce the following remarkable notion, demonstrating the interplay
of adjacent and transpose angles.

\begin{definition}\label{T_i}
Let $\alpha$ be a non-tivial rational angle.
Set
$$
T(\alpha)=
\left\{
\begin{array}{ll}
\pi-E_{1}(\alpha^t), \quad &\text{if } \chi_1(\alpha^t)>1;\\
\big(E_{1}(\pi-\alpha)\big)^t, \quad &\text{if } \chi_1(\pi-\alpha)>1.
\end{array}
\right.
$$
\end{definition}

\begin{remark}
Note that, since $\alpha^t$ reverses the odd continued fraction expansion of $\alpha$ and $\pi-\alpha$ reverses the even continued fraction expansion, exactly one of $\chi_1(\alpha^t)$ or $\chi_1(\pi-\alpha)$ is greater than $1$.
\end{remark}

Now all the strong best approximations can be written in terms of rational angles as follows.

\begin{proposition}\label{adj-trans}
Let $a$ be a rational number and let $m$ be length of the continued fraction expansion $($such that $a_m>1)$.
Then, the set of strong best approximations of $a$ is as follows
$$
\Big\{
\itan \big((T)^k(\iarctan a)\big)\Big| k=0,1,\ldots, m-1
\Big\}.
$$
\end{proposition}

\begin{proof}
Set $\alpha=\iarctan a$.
The  proof is straightforward as $\alpha^t$ reverses the odd regular continued fraction for $a$
and $\pi-\alpha$ reverses the even regular continued fraction for $a$.
\end{proof}

\begin{example}
Let
$$
a=[1;2,3,4]=[1;2,3,3,1]=\frac{43}{30}.
$$
In our case, it is the even continued fraction expansion that has a final element greater than $1$, so we aim to find $T(\alpha)=(E_1(\pi-\alpha))^t$.
Taking
$\beta=\pi- \alpha$, we see
$$
\tan \beta =[4;3,2,1]=\frac{43}{10}.
$$
Therefore,
$$
\IARCTAN(\beta)=
\left(
\begin{array}{cc}
1& 10\\
0& 43\\
\end{array}
\right).
$$
and
$$
\IARCTAN (E_{1}(\beta))=
\left(
\begin{array}{cc}
1& 3\\
0& 10\\
\end{array}
\right).
$$
The odd regular continued fraction for $10/3$ is
$[3;2,1]$.
To find $T(\alpha)$ it remains to take the transpose of $\iarctan (10/3)$. Then
$$
\itan(T(\alpha))=\itan(\iarctan^t(10/3))=[1;2,3]=\frac{10}{7}.
$$
This constructs the second to last strong best approximation for $43/30$ (the last being $43/30$).

Iteratively, we can construct the other two strong best approximations: $3/2=[1;2]$ and $1=[1]$.
We omit the computations here.
\end{example}

\subsubsection{Higher dimensional generalisations}
Using the above framework, we can generalise the notion of a strong best approximation of a rational angle to include rational cones.

\begin{definition}
Let $\alpha$ be a non-trivial rational cone, let $i\in \{1,\ldots, k-1\}$,
and let $\icos_{i,k}\alpha \ge 2$.
Set
$$
T_i(\alpha)=
\left\{
\begin{array}{ll}
\pi_i-E_i(\alpha_{(i,k)}), \quad &\text{if } \chi_i(\alpha_{(i,k)})>1;\\
\big(E_i(\pi_i-\alpha)\big)_{(i,k)}, \quad &\text{if } \chi_i(\pi_i-\alpha)>1.
\end{array}
\right.
$$
\end{definition}

\begin{remark}As for Definition~\ref{T_i}, exactly one of $\chi_i(\alpha_{(i,k)})>1$  or $\chi_i(\pi_i-\alpha)>1$ can occur.
\end{remark}

We use Proposition~\ref{adj-trans} to justify the below definition of strong best approximations for cones.

\begin{definition}
We say that a rational cone $\beta$ is a strong best approximation of a rational cone $\alpha$ if
there exist a finite sequence $(i_k)$ and a permutation $s\in S_n$ such that
$\beta$ is integer congruent to
$$
T_{i_k}\circ\cdots\circ T_{i_1}(\alpha_s),
$$
where $\circ$ denotes the map composition.
\end{definition}

We immediately get the following interesting problem.

\begin{problem}
Study the approximation properties for these strong best approximations.
(Namely, what could be the analogue for the above Equation~\eqref{strong-best-classic}.)
\end{problem}

\section{Proofs of Section~3}\label{Sec4}

\subsection{Canonical coordinates and points}

Let us start with the definitions of two important coordinate systems which are uniquely defined by rational cones.

\begin{definition}
Let $\alpha=\angle (A_0;A_1\ldots A_k)$ be an ordered rational cone in $\r^n$.
Consider the coordinate system with the origin at $A_0$
and the integer lattice basis $e_1,\ldots e_k$ such that $\alpha$ coincides with $\iarctan \alpha$
in this basis.
This coordinate system is said to be {\it canonical} for $\alpha$.
\end{definition}

Note that the canonical coordinate system is uniquely defined by a cone.

\begin{definition}
Consider a rational cone $\alpha$.
The \textit{canonical point} of the cone, denoted $c(\alpha)$, is the point whose canonical coordinates equal
$$
(1,1,\ldots, 1).
$$
\end{definition}

The second important coordinate system associated to a cone
is defined by the vectors of its edges.

\begin{definition}
Consider a cone written in the form $\alpha=(A_0;A_0+v_1\ldots A_0+v_k)$. Then $\iv(\alpha)=|\det(v_1,\ldots, v_k)|.$
The {\it $\alpha$-coordinates} of an integer point
$p$ in the span of $\alpha$ are the coefficients $(\lambda_1,\ldots, \lambda_k)$
defined from the identity
\begin{equation*}
p=\frac{1}{\iv(\alpha)}\cdot\big(\lambda_1v_1+\cdots +\lambda_kv_k\big).
\end{equation*}
\end{definition}

Note that the integer nodes of $\alpha$-coordinates
are not necessarily integer points in the original integer sublattice of $\z^n$.

A nice property of both canonical coordinates and $\alpha$-coordinates of points is that they
are invariant under integer congruences.

It turns out that the $\alpha$-coordinates of the canonical points identify the integer cosines of simple cones.
Namely, the following statement holds.

\begin{proposition}\label{propp}
Let $\alpha$ be a simple rational cone.
Then the $\alpha$-coordinates of $c(\alpha)$ are as follows
$$
c(\alpha)\equiv (-\icos_{1,k}\alpha,\ldots, -\icos_{k-1,k}\alpha, 1 )   \mod \iv(\alpha).
$$
\end{proposition}

\begin{proof}
This proposition follows immediately by considering
$$
\iv(\alpha)\cdot (\IARCTAN(\alpha))^{-1}(1,1,\ldots, 1).
$$\end{proof}

Proposition~\ref{propp} shows that the $\alpha $-coordinates of the canonical point uniquely (and explicitly)
determine all the integer cosines of $\alpha$.

\subsection{Transpose angles: proof of Proposition~\ref{transpose-highdim}}\label{Transpose angles -- proof}

Consider a simple rational cone $\alpha=(A_0;A_0+v_1\ldots A_0+v_k)$.
If we swap two vectors (excluding $v_k$),
then this corresponds to swapping the integer cosines in the $\alpha$-coordinates of the canonical point.

Now consider the case when we swap $v_i$ with $v_k$ (for some $i<k$).
If we set
$$
s \equiv -(\icos_{i,k}\alpha)^{-1} \mod \iv(\alpha),
$$
then  there is a unique integer point with the
$i$-th $\alpha$-coordinate equalling $1$, given by
$$
s\cdot  c(\alpha)\mod \iv(\alpha).
$$
Therefore, $s\cdot c(\alpha)\mod \iv(\alpha)$ are the $\alpha$-coordinates for the canonical point $c(\alpha_{(i,n)})$.

Comparing the canonical points of these transpose cones provides us with a proof of Proposition~\ref{transpose-highdim}.
\qed

\subsection{Adjacent  angles: proof of Proposition~\ref{adjacent-highdim}}\label{Adjacent angles -- proof}
Let $\alpha$ be a simple rational cone and let $i\in\{1,2,\ldots,k\}$.
The  cone $\pi_i -\alpha$ is obtained from $\alpha$ by reversing the sign of $v_i$.
Therefore, the canonical point for $\pi_i -\alpha$ has canonical coordinates for $\alpha$ given by
$$
(1,\ldots, 1,0,1,\ldots,1)
$$
(with $0$ in the $i$-th position).
Therefore, the $\alpha$-coordinates for the point $c(\pi_i -\alpha)$ are
$$
\begin{array}{l}
\iv(\alpha)\cdot (\IARCTAN(\alpha))^{-1}(1,\ldots, 1,0,1,\ldots,1)
\equiv\\
\qquad \qquad \iv(\alpha)\cdot (\IARCTAN(\alpha))^{-1}(1,\ldots, 1,1,1,\ldots,1)
\mod\iv(\alpha).
\end{array}
$$
Since we are working with simple rational cones, we have  $\isin_k(\alpha)=\iv(\alpha)$.
To switch to the $(\pi_i-\alpha)$-coordinates of  $c(\pi_i -\alpha)$,
we reverse the sign of the $i$-th $\alpha$-coordinate.
Therefore, we have
$$
\icos_{j,k}(\pi_i-\alpha)\equiv
\left\{
\begin{array}{lll}
 -\icos_{j,k}(\alpha) &\mod\isin_k(\alpha), \quad &\hbox{if $i\ne j$;}\\
 \icos_{j,k}(\alpha) &\mod\isin_k(\alpha), \quad &\hbox{if $i=j$.}\\
\end{array}
\right.
$$

Finally, if $i=k$, the canonical coordinates for $\alpha$ of the point $c(\pi-\alpha)$ are
$$
c(\pi-\alpha)=(0,\ldots,0, -1).
$$
The $\alpha$ coordinates of this point are
$$
\begin{array}{l}
\iv(\alpha)\cdot (\IARCTAN(\alpha))^{-1}(0,\ldots,0,-1)
\equiv\\
\qquad \qquad - \iv(\alpha)\cdot (\IARCTAN(\alpha))^{-1}(1,\ldots,1)
\mod\iv(\alpha).
\end{array}
$$
Then the computations repeat, as above.
This completes the proof of Proposition~\ref{adjacent-highdim}. $\qed$

\subsection{Angles generating the same simplex: proof of Proposition~\ref{simplex}}\label{Simplices -- proof}

This proof uses the same ideas as the previous two proofs.
Write the cone $\alpha$ in the form $\alpha=\angle(A_0; A_0+v_1\ldots A_0+v_k)$.
Without loss of generality, assume that the vector between the vertices of $\beta$ and $\alpha$ is $v_1$.
Since all the vectors of the original simplex are of unit integer lengths,
$\beta$ is generated by the vectors $(-v_1,v_2-v_1,v_3-v_1,\ldots, v_k-v_1)$.
The canonical point of $\beta$ in the canonical coordinates for $\alpha$ is as follows
$$
c(\beta)=(-1,1,1,\ldots, 1).
$$
The rest of the computation follows as in the previous proofs.
\qed

\subsection{Proof of Proposition~\ref{Determinant_of_Special_Matrix}}\label{DSM}

By Corollary~\ref{strictcycle}, for $i\in\{2,\ldots,k\}$ the $i$-th row of $M_\alpha$  can be written as
$$
-\icos_{i,k}^{-1}\alpha\cdot(-1,\icos_{1,k} \alpha, \ldots, \icos_{i-1,k} \alpha, 0, \icos_{i+1,k}\alpha, \ldots, \icos_{k-1,k}\alpha)\mod \iv(\alpha).
$$
By pulling out the common row factors, the determinant of $M_\alpha$ is equal to
$$
\left(\prod_{i=1}^{k-1}-\icos_{i,k}^{-1}\alpha \right)\det \left(
\begin{array}{ccccc}
0&\icos_{1,k} \alpha & \icos_{2,k} \alpha & \ldots & \icos_{k-1,k}\alpha \\
-1&  0  & \icos_{2,k} \alpha &\ldots&\icos_{k-1,k} \alpha\\
-1&\icos_{1,k}\alpha &   0  &\ldots&\icos_{k-1,k} \alpha \\
\vdots&\vdots& \vdots &\ddots&\vdots\\
-1 & \icos_{1,k} \alpha &\ldots&\icos_{k-2,k} \alpha&\hspace{9mm} 0 \hspace{9mm} \\
\end{array}
\right)\mod \iv(\alpha).
$$

Factoring out the common column factors, the determinant is
$$
(-1)\left(\prod_{i=1}^{k-1}-\icos_{i,k}^{-1}\alpha \right)\left(\prod_{i=1}^{k-1}\icos_{i,k}\alpha \right)\det \left(
\begin{array}{cccc}
0&1  &  \ldots & 1\\
 1&  0 &\ddots&\vdots\\
\vdots&\ddots&\ddots&1\\

1 &\ldots&1&0\\
\end{array}
\right)\mod \iv(\alpha).
$$
Finally, by cancelling out the products, the determinant of $M_\alpha$ is
$$
(-1)^{k}\cdot\det \left(
\begin{array}{cccc}
0&1  &  \ldots & 1\\
 1&  0 &\ddots&\vdots\\
\vdots&\ddots&\ddots&1\\

1 &\ldots&1&0\\
\end{array}
\right)\equiv (-1)^k\cdot(-1)^k\cdot(1-k)\equiv 1-k \mod \iv(\alpha),
$$
as required.
\qed

\section{A few words on non-simple cones}\label{A few words on general lattices}

Let $\alpha$ be a rational cone.
Consider $\IARCTAN(\alpha)$ and remove several last rows until we have a $(k-1)\times k$ matrix.
Let us take all its minors and denote them by $p_i(\alpha)$
for $i=1,\ldots k$.
Then the coordinates $(p_1(\alpha),\ldots, p_i(\alpha))$ are the {\it Pl\"uckers coordinates}  of the tangent.

Proposition~\ref{transpose-highdim} can now be written in terms of Pl\"uckers coordinates as follows.

\begin{proposition}\label{transpose-general}
Consider a rational cone $\alpha$  and consider its transpose angle  $\alpha_{(i,k)}$ for some $i\in\{1,2,\ldots,k-1\}$.
Then we have the following identities
$$
p_j(\alpha)\cdot p_j(\alpha_{(i,k)})\equiv \frac{\iv(\alpha)}{\isin_{k}\alpha}\cdot \frac{\iv(\alpha)}{\isin_{k}\alpha_{(i,k)}} \mod \iv(\alpha).
$$
\end{proposition}

\begin{remark}
The proof in this case repeats the proof of Proposition~\ref{transpose-highdim}
(via reduction to an analogue of the expression in Proposition~\ref{propp}).
So we omit it here.
\end{remark}

Note also that
$$
\frac{\iv(\alpha)}{\isin_{k}\alpha}=\prod\limits_{i=1}^{k-1}\isin_{i}\alpha.
$$

\begin{example}
Consider a rational $3$-cone $\alpha$ in $\r^3$.
Then the Pl\"ucker coordinates for $\IARCTAN(\alpha)$ are as follows
$$
\begin{array}{l}
p_1(\alpha)= \icos_{1,2}\alpha \cdot\icos_{2,3}\alpha-\isin_2\alpha \cdot \icos_{1,3} \alpha;\\
p_2(\alpha)= -\icos_{2,3} \alpha;\\
p_3(\alpha)=\isin_2\alpha.
\end{array}
$$
Therefore, the following relations hold for the transpose angle $\alpha_{(1,3)}$
$$
p_i(\alpha)\cdot p_i(\alpha_{(1,3)}) \equiv \isin_2 \alpha \cdot \isin_2 \alpha_{(1,3)} \mod \iv(\alpha).
$$
%Recall that
%$
%\iv(\alpha)=\isin_2\alpha \cdot\isin_3\alpha=\isin_2\alpha_{(13)} \cdot\isin_3\alpha_{(13)}
%$.
\end{example}

\begin{remark}
We expect that similar formulae can be deduced via Pl\"ucker coordinates for adjacent cones and
the cones which generate the same simplices.
\end{remark}

{\noindent
{\bf Acknowledgements.}
The first and the last authors are partially supported by the EPSRC grant EP/W006863/1.
}

{\small\bibliography{commat}}

\EditInfo{February 7, 2023}{March 29, 2023}{Camilla Hollanti and Lenny Fukshansky}

\end{document}